\documentclass[12pt]{amsart}
\usepackage{amsmath,amssymb,amsthm,amscd,verbatim}
\usepackage{graphicx}
\usepackage{hyperref}
\usepackage{enumerate}

\usepackage{xcolor}
\hypersetup{
    colorlinks,
    linkcolor={blue},
    citecolor={red},
    urlcolor={blue}
}

\newtheorem{thm}{Theorem}
\newtheorem{lem}[thm]{Lemma}
\newtheorem{conj}[thm]{Conjecture}

\newtheorem{cor}[thm]{Corollary}

\def\longequation{$$\vcenter\bgroup\advance\hsize by -9em%
\noindent\ignorespaces\refstepcounter{equation}}%
\makeatletter%
\def\endlongequation{\egroup\eqno(\theequation)$$\global\@ignoretrue}
\makeatother

\newcommand{\cB}{\mathcal{B}}

\newcommand{\cC}{\mathcal{C}}

\renewcommand{\le}{\leqslant}

\renewcommand{\ge}{\geqslant}

\newcommand{\nzf}{{\sc nzf}}
\newcommand{\asf}{{\sc asf}}
\newcommand{\bZ}{\mathbb{Z}}

\begin{document}
\title{Additive bases and flows in graphs}
\author[L. Esperet]{Louis Esperet} \address{Laboratoire G-SCOP (CNRS,
  Univ. Grenoble-Alpes), Grenoble, France}
\email{louis.esperet@grenoble-inp.fr}

\author[R. de Joannis de Verclos]{R\'emi de Joannis de Verclos}\address{Laboratoire G-SCOP (CNRS,
  Univ. Grenoble-Alpes), Grenoble, France}
\email{remi.de-joannis-de-verclos@grenoble-inp.fr}

\author[T.-N. Le]{Tien-Nam Le}\address{Laboratoire d'Informatique du
  Parall\'elisme, \'Ecole Normale Sup\'erieure de Lyon, France}
  \email{tien-nam.le@ens-lyon.fr}

\author[S. Thomass\'e]{St\'ephan Thomass\'e}\address{Laboratoire d'Informatique du
  Parall\'elisme, \'Ecole Normale Sup\'erieure de Lyon, France}
  \email{stephan.thomasse@ens-lyon.fr}

\thanks{The authors are partially supported by ANR Project STINT
  (\textsc{anr-13-bs02-0007}) and GATO (\textsc{anr-16-ce40-0009-01}), and LabEx PERSYVAL-Lab
  (\textsc{anr-11-labx-0025}). An extended abstract of this work
  appeared in the proceedings of Eurocomb 2017.}

\date{}
\sloppy

\begin{abstract}
It was conjectured by Jaeger, Linial, Payan, and Tarsi in 1992
that for any prime number $p$, there is a constant $c$ such that for
any $n$, the union (with repetition) of the vectors of any family of
$c$ linear bases of $\bZ_p^n$
forms an additive basis of $\bZ_p^n$ (i.e. any element of $\bZ_p^n$
can be expressed as the sum of a subset of these vectors).
In this note, we prove this
conjecture when each vector contains at most two non-zero
entries. As an application, we prove several results on flows in
highly edge-connected graphs, extending known
results. For instance, assume that $p\ge 3$ is a prime number and $\vec{G}$ is a directed, highly
edge-connected graph in which each arc is given a list of two distinct 
values in $\bZ_p$. Then $\vec{G}$ has a $\bZ_p$-flow in which each arc
is assigned a value of its own list.
\end{abstract}
\maketitle

\section{Introduction}
Graphs considered in this paper may have multiple edges but no loops.
An \emph{additive basis} $B$ of a vector space $F$ is a multiset of elements from $F$ such that for all $\beta\in F$, there is a subset of $B$ which
sums to $\beta$. Let $\bZ_p^n$ be the $n$-dimensional linear space over the prime field
$\bZ_p$.
The following result is a simple consequence of
the Cauchy-Davenport Theorem~\cite{Dav35} (see also~\cite{ANR96}).

\begin{thm}[\cite{Dav35}]\label{thm:CD}
For any prime $p$, any multiset of $p-1$ non-zero elements of $\bZ_p$ forms an additive
basis of $\bZ_p$.
\end{thm}

This result can be rephrased as: \emph{for $n=1$, any family of $p-1$ linear
bases of $\bZ_p^n$ forms an additive
basis of $\bZ_p^n$}. A natural question is whether this can be
extended to all integers $n$. 
Given a collection of sets $X_1,...,X_k$, we denote by $\biguplus_{i=1}^{k}X_i$ the union with repetitions of $X_1,...,X_k$.
Jaeger, Linial, Payan and
Tarsi~\cite{Jae92} conjectured the following, a generalization of
important results regarding nowhere-zero flows in graphs.

\begin{conj}[\cite{Jae92}]\label{conj:matrix}
For every prime number
$p$, there is a constant $c(p)$ such that for any $t\ge c(p)$ linear bases 
$B_1,...,B_t$
of $ \bZ_p^n$, the union $\biguplus_{s=1}^{t}B_s$ forms an additive basis of $ \bZ_p^n$.
\end{conj}

Alon, Linial and Meshulam \cite{ALM91} proved a weaker version of
Conjecture~\ref{conj:matrix}, that the union of any $p\lceil\log
n\rceil$ linear bases of $ \bZ_p^n$ contains an additive basis of $
\bZ_p^n$ (note that their bound depends on $n$). 
The \emph{support} of a vector $x=(x_1,\ldots,x_n)\in \bZ_p^n$ is the
set of indices $i$ such that $x_i\ne 0$. The \emph{shadow} of a vector
$x$ is
the (unordered) multiset of
non-zero entries of $x$. 
Note that sizes of the support and of the shadow of a vector are equal. 
In this note, we prove that Conjecture~\ref{conj:matrix}
holds if the support of each vector has size at most two.

\begin{thm}\label{thm:matrix}
Let $p\ge 3$ be a prime number. For some integer $\ell \ge 1$, consider $t\ge 8 \ell (3p-4) +p-2$ linear bases 
$B_1,...,B_t$ 
of $ \bZ_p^n$, such that the support of each vector has size at most $2$, and
at most $\ell$ different shadows of size $2$ appear among the vectors of
$\cB=\biguplus_{s=1}^{t}B_s$. Then $\cB$ forms an additive basis of $ \bZ_p^n$.
\end{thm}

Theorem~\ref{thm:matrix} will be proved in
Section~\ref{sec:matrix} using a result of Lov\'asz, Thomassen,
Wu and Zhang~\cite{LTWZ13} (Theorem~\ref{thm:boundary} below) on flows
in highly edge-connected graphs.  It was mentioned to us by one of the
referees that Lai and Li~\cite{LaiLi} established the equivalence
between Theorem~\ref{thm:boundary} and Theorem~\ref{thm:matrix} in the
special case where all the shadows are equal
to $\{-1,+1\} \pmod p$.

The number of possibilities for an
(unordered) multiset of $\bZ_p\setminus\{0\}$ of size 2 is ${p-1 \choose 2}+p-1={p \choose 2}$. As a consequence,
Theorem~\ref{thm:matrix} has the following immediate corollary.

\begin{cor}\label{cor:matrix}
Let $p\ge 3$ be a prime number. For any $t\ge 8 {p \choose 2} (3p-4) +p-2$ linear bases 
$B_1,...,B_t$ 
of $ \bZ_p^n$ such that the support of each vector has size at most 2, $\biguplus_{s=1}^{t}B_s$ forms an additive basis of $ \bZ_p^n$.
\end{cor}

Another interesting consequence of Theorem~\ref{thm:matrix} concerns
the linear subspace $(\bZ_p^n)_0$ of vectors of $\bZ_p^n$ whose
entries sum to $0 \pmod p$.

\begin{cor}\label{cor:zerosum}
Let $p\ge 3$ be a prime number. For any $t\ge 4 (p-1)(3p-4) +p-2$ linear bases 
$B_1,...,B_t$ 
of $(\bZ_p^n)_0$ such that the support of each vector has size at most 2, $\biguplus_{s=1}^{t}B_s$ forms an additive basis of $ (\bZ_p^n)_0$.
\end{cor}

\begin{proof}
Note that for any $1\le s \le t$, the linear basis $B_s$ consists of $n-1$
vectors, each of which has a support of size 2, and the two elements of
the shadow sum to $0 \pmod p$. In particular, at most $\tfrac{p-1}2$
different shadows appear among the vectors of the linear bases 
$B_1,...,B_t$. It is convenient to view each $B_s$
as a matrix in which the elements of the basis are column vectors.
For each $1\le s \le t$, let $B_s'$
be obtained from $B_s$ by deleting the last row. It is easy to see
that $B_s'$ is a linear basis of $\bZ^{n-1}_p$. Moreover, at most $\tfrac{p-1}2$
different shadows of size 2 appear among the vectors of the linear bases 
$B'_1,...,B'_t$ (note that the removal of the
last row may have created
vectors with shadows of size 1).
In
particular, it follows from Theorem ~\ref{thm:matrix} that for any vector $\beta=(\beta_1,\ldots,\beta_n) \in
(\bZ_p^n)_0$, the vector $(\beta_1,\ldots,\beta_{n-1}) \in
\bZ_p^{n-1}$ can be written as a sum of a subset of elements of
$\biguplus_{s=1}^{t}B_s'$. Clearly, the corresponding
subset of elements of $\biguplus_{s=1}^{t}B_s$ sums to $\beta$. This
concludes the proof of Corollary~\ref{cor:zerosum}.
\end{proof}

In the next section, we explore some consequences of Corollary~\ref{cor:zerosum}.

\section{Orientations and flows in graphs}

Let $G=(V,E)$ be a non-oriented graph. An \emph{orientation} $
\vec{G}=(V,\vec{E})$ of $G$ is
obtained by giving each edge of $E$ a
direction. For each edge $e\in E$, we denote the corresponding arc of
$\vec{E}$ by $\vec{e}$, and vice versa. For a vertex $v\in V$, we denote by
$\delta^+_{\vec{G}}(v)$ the set of arcs of $\vec{E}$ leaving $v$, and by $\delta^-_{\vec{G}}(v)$
the set of arcs of $\vec{E}$ entering $v$.  

For an integer $k\ge 2$, a mapping $\beta : V \rightarrow \bZ_{k}
$ is said to be a \emph{$\bZ_{k}$-boundary} of $G$ if $\sum_{v \in V}
\beta(v)\equiv 0 \pmod {k}$.
Given a $\bZ_{k}$-boundary $\beta$ of
$G$, an orientation $\vec{G}$ of $G$ is a \emph{$\beta$-orientation}
if $d_{\vec{G}}^+(v)-d_{\vec{G}}^-(v)\equiv \beta(v) \pmod
{k}$ for every $v\in V$, where $d_{\vec{G}}^+(v)$ and $d_{\vec{G}}^-(v)$ stand for the out-degree and the
in-degree of $v$ in $\vec{G}$.

The following major result was obtained by  Lov\'asz, Thomassen,
Wu and Zhang~\cite{LTWZ13}:

\begin{thm}{\cite{LTWZ13}}\label{thm:boundary}
For any $k \ge 1$, any $6k$-edge-connected graph $G$, and any
$\bZ_{2k+1}$-boundary $\beta$ of $G$, the graph $G$ has a $\beta$-orientation.
\end{thm}

A natural question is whether a weighted counterpart of
Theorem~\ref{thm:boundary} exists. 
Given a graph $G=(V,E)$,
%or every orientation $\vec{G}=(V,\vec{E})$ of $G$ and every mapping $f: \vec{E}
%\rightarrow \bZ_{2k+1} $, we define $\partial f(v)=\sum_{\vec{e} \in \delta^+(v)} f(\vec{e})-\sum_{\vec{e} \in
%  \delta^-(v)}f(\vec{e})$ for every $v \in V$. 
a $\bZ_{k}$-boundary $\beta$ of
$G$ and a mapping $f: E
\rightarrow \bZ_{k} $, an orientation $\vec{G}$ of $G$ is called
an \emph{$f$-weighted $\beta$-orientation}
if $\partial f(v) \equiv \beta(v) \pmod
{k}$ for every $v$, 
where $\partial f(v)=\sum_{\vec{e} \in \delta^+_{\vec{G}}(v)} f({e})-\sum_{\vec{e} \in
 \delta^-_{\vec{G}}(v)}f({e})$. 
Note that if
$f(e)\equiv 1 \pmod{k}$ for every edge $e$, an $f$-weighted
$\beta$-orientation is precisely a $\beta$-orientation. 

\smallskip

An immediate observation is that if we wish to have a general result of the form of
Theorem~\ref{thm:boundary} for weighted orientations, it is necessary to
assume that $2k+1$ is a prime number. For instance, take
$G$ to consist of two vertices $u,v$ with an arbitrary number of edges
between $u$ and $v$, consider a non-trivial divisor $p$ of $2k+1$, and ask
for a $\mathbf p$-weighted $\bZ_{2k+1}$-orientation $\vec{G} $ of $G$ 
(here, $\mathbf p$ denotes the function that maps each edge to $p \pmod
{2k+1}$). Note
that for any orientation, $\partial {\mathbf p}(v)$ is in the
subgroup of $\bZ_{2k+1}$ generated by $p$, and this subgroup does
not contain $1,-1 \pmod
{2k+1}$. In particular, there is no $\mathbf p$-weighted
$\bZ_{2k+1}$-orientation of $G$ with boundary
$\beta$ satisfying $\beta(u)\equiv-\beta(v)\equiv 1 \pmod {2k+1}$.

\medskip

In Section~\ref{sec:flow}, we will prove that Corollary~\ref{cor:zerosum}
easily implies a weighted counterpart of
Theorem~\ref{thm:boundary} as in the following theorem, but with a
stronger requirement on the
edge-connectivity. Theorem~\ref{thm:lyon} itself will be deduced directly
from Theorem~\ref{thm:boundary}.

\vspace{8pt}

\begin{thm}~\label{thm:lyon}
Let $p\ge 3$ be a prime number and let $G=(V,E)$ be a
$(6p-8)(p-1)$-edge-connected graph. For any mapping $f: E
\rightarrow \bZ_{p}\setminus \{0\} $ and any $\bZ_{p}$-boundary
$\beta$, $G$ has an $f$-weighted $\beta$-orientation.
\end{thm}

Theorem~\ref{thm:lyon} turns out to be equivalent to the following seemingly more general
result. Assume that we are given a directed graph
$\vec{G}=(V,\vec{E})$ and a $\bZ_{p}$-boundary
  $\beta$. A \emph{$\bZ_{p}$-flow} with {boundary
  $\beta$} in $\vec{G}$ is a mapping $f: \vec{E} \rightarrow \bZ_{p} $ such that
$\partial f(v)\equiv \beta(v) \pmod p$ for every $v$. In other words,
$f$ is a $\bZ_{p}$-flow with boundary
  $\beta$ in $\vec{G}=(V,\vec{E})$ if and only if $\vec{G}$ is an $f$-weighted
  $\beta$-orientation of its underlying non-oriented graph $G=(V,E)$, where
  $f$ is extended from $\vec{E}$ to $E$ in the natural way (i.e. for
  each $e\in E$, $f(e):=f(\vec{e})$). 

In the remainder
of the paper we
will say that a directed graph $\vec{G}$ is \emph{$t$-edge-connected} if its
underlying non-oriented graph, denoted by $G$, is $t$-edge-connected.

\vspace{8pt}

\begin{thm}~\label{thm:choos}
Let $p\ge 3$ be a prime number and let $\vec{G}=(V,\vec{E})$ be a
directed $(6p-8)(p-1)$-edge-connected graph. For any arc $\vec{e} \in
\vec{E}$, let $L(\vec{e})$ be a pair of distinct elements of $\bZ_{p}$.
Then for every $\bZ_{p}$-boundary
$\beta$, $\vec{G}$ has a $\bZ_{p}$-flow $f$ with boundary
$\beta$ such that for any $\vec{e} \in
\vec{E}$, $f(\vec{e}) \in L(\vec{e}) $.
\end{thm}

This result can been seen
as a choosability version of Theorem~\ref{thm:boundary} (the reader is referred
to~\cite{DeV00} for choosability versions of some
classical results on flows). 
To see that Theorem~\ref{thm:choos} implies Theorem~\ref{thm:lyon},
simply fix an arbitrary orientation of $G$ and set
$L(\vec{e})=\{f(e),-f(e)\}$ for each arc $\vec{e}$. We now prove that
Theorem~\ref{thm:lyon} implies Theorem~\ref{thm:choos}. We actually
prove a slightly stronger statement (holding in $\bZ_{2k+1}$
for any integer $k\ge 1$).

\begin{lem}\label{lem:impl}
Let $k\ge 1$ be an integer, and 
let $\vec{G}=(V,\vec{E})$ be a directed graph such that the underlying
non-oriented graph $G$ has an $f$-weighted $\beta$-orientation for any mapping $f: E
\rightarrow \bZ_{2k+1}\setminus \{0\} $ and any $\bZ_{2k+1}$-boundary
$\beta$. For every arc $\vec{e} \in
\vec{E}$, let $L(\vec{e})$ be a pair of distinct elements of $\bZ_{2k+1}$. Then for every $\bZ_{2k+1}$-boundary
$\beta$, $\vec{G}$ has a $\bZ_{2k+1}$-flow $g$ with boundary
$\beta$ such that $g(\vec{e}) \in L(\vec{e}) $ for every $\vec{e}$.
\end{lem}

\begin{proof}
Let $\beta$ be a $\bZ_{2k+1}$-boundary of $\vec{G}$. Consider a single
arc $\vec{e}=(u,v)$ of $\vec{G}$. Choosing one of the two values of
$L(\vec{e})$, say $a$ or $b$, will either add $a$ to $\partial g(u)$
and subtract $a$ from $\partial g(v)$, or add $b$ to $\partial g(u)$
and subtract $b$ from $\partial g(v)$. Note that 2 and $2k+1$ are
relatively prime, so the element $2^{-1}$ is well-defined in
$\bZ_{2k+1}$. If we now add $2^{-1}(a+b)$ to $\beta(v)$ and subtract
$2^{-1}(a+b)$ from $\beta(u)$, the earlier choice is equivalent to
choosing between the two following options: adding $2^{-1}(a-b)$ to $\partial g(u)$
and subtracting $2^{-1}(a-b)$ from $\partial g(v)$, or adding $2^{-1}(b-a)$ to $\partial g(u)$
and subtracting $2^{-1}(b-a)$ from $\partial g(v)$. This is equivalent to
choosing an orientation for an edge of weight $2^{-1}(a-b)$. It
follows that finding a $\bZ_{2k+1}$-flow $g$ with boundary
$\beta$ such that for any $\vec{e} \in
\vec{E}$, $g(\vec{e}) \in L(\vec{e}) $ is equivalent to finding an
$f$-weighted $\beta'$-orientation for some other
$\bZ_{2k+1}$-boundary $\beta'$ of $G$, where the weight $f(e)$ of each edge $e$ is $2^{-1}$ times the
difference between the two elements of $L(\vec{e})$.
\end{proof}

We now consider the case where $L(\vec{e})=\{0,1\}$ for every arc
$\vec{e}\in \vec{E}$. Let $f_{2^{-1}}: \vec{E}\rightarrow \bZ_{2k+1}$ denote the
function that maps each arc $\vec{e}$ to $2^{-1} \pmod {2k+1}$. 
The same argument as in the proof of Lemma \ref{lem:impl} implies that if ${G}$ has an $f_{2^{-1}}$-weighted $\beta$-orientation for every $\bZ_{2k+1}$-boundary
$\beta$, then for every $\bZ_{2k+1}$-boundary
$\beta$, the digraph $\vec{G}$ has a $\bZ_{2k+1}$-flow $f$ with boundary
$\beta$ such that $f(\vec{e}) \in L(\vec{e}) $ for every $\vec{e}$.

The following is a simple corollary of
Theorem~\ref{thm:boundary}.

\begin{cor}\label{cor:thomexp2}
Let $\ell \ge 1$ be an odd integer and let $k\ge 1$ be relatively prime
with $\ell$.  Let $G=(V,E)$ be a
$(3\ell-3)$-edge-connected graph, and let ${\mathbf k} : E
\rightarrow \bZ_{\ell} $ be the mapping that assigns $k \pmod \ell$ to each
edge $e\in E$. Then for any $\bZ_{\ell}$-boundary
$\beta$, $G$ has a $\mathbf k$-weighted $\beta$-orientation.
\end{cor}

\begin{proof}
Observe that $\beta'=k^{-1} \cdot \beta$ is a $\bZ_{\ell}$-boundary
($k^{-1}$ is well defined in $\bZ_{\ell}$). It
follows from Theorem~\ref{thm:boundary} that $G$ has a
$\beta'$-orientation. Note that this corresponds to a $\mathbf
k$-weighted $\beta$-orientation of $G$, as desired.
\end{proof}

As a consequence, the following is an equivalent version of
Theorem~\ref{thm:boundary} (see also~\cite{Jae92,LaiLi}).

\begin{thm}\label{thm:thomexp1}
Let $k\ge 1$ be an integer and let $\vec{G}=(V,\vec{E})$ be a
directed $6k$-edge-connected graph. Then for every $\bZ_{2k+1}$-boundary
$\beta$, $\vec{G}$ has a $\bZ_{2k+1}$-flow $f$ with boundary
$\beta$ such that $f(\vec{E}) \in \{0, 1\} \pmod {2k+1}$.
\end{thm}

This version of Theorem~\ref{thm:boundary} will allow us to derive interesting results on antisymmetric flows in directed
highly edge-connected graphs. Given an abelian group $(B, +)$,
a \emph{$B$-flow} in $\vec{G}$ is a mapping $f: \vec{E} \rightarrow B$ such that  $\partial f(v)=0$ for
every vertex $v$, where all operations are
performed in $B$. A $B$-flow
$f$ in $\vec{G}=(V,\vec{E})$
is a \emph{nowhere-zero $B$-flow} (or a $B$-\nzf) if
$0\not\in f(\vec{E})$, i.e. each arc of $\vec{G}$ is assigned a non-zero element of
$B$. If no two arcs receive inverse	
elements of $B$, then $f$ is an
\emph{antisymmetric $B$-flow} (or a $B$-\asf).

Since $0=-0$, a $B$-\asf\ is also a $B$-\nzf. It was conjectured by
Tutte that every directed 2-edge-connected graph has a
$\bZ_5$-\nzf~\cite{Tut54}, and that every directed 4-edge-connected
graph has a $\bZ_3$-\nzf~(see~\cite{Ste76} and~\cite{BM76}). 
Antisymmetric flows were introduced by Ne\v{s}et\v{r}il and Raspaud
in~\cite{NR99}. A natural obstruction for the existence of an
antisymmetric flow in a directed graph $\vec{G}$ is the presence of directed
2-edge-cut in $\vec{G}$. Ne\v{s}et\v{r}il and Raspaud asked whether any
directed graph without directed 2-edge-cut has a $B$-\asf, for some
$B$. This was proved by DeVos,  Johnson, and Seymour in~\cite{DJS},
who showed that any
directed graph without directed 2-edge-cut has a $\bZ_2^8 \times
\bZ_3^{17}$-\asf. It was later proved by DeVos, Ne\v{s}et\v{r}il, and
Raspaud~\cite{DNR04}, that the group could be replaced by $\bZ_2^6 \times
\bZ_3^{9}$. The best known result is due to Dvo\v{r}\'ak, Kaiser, Kr\'al', and
  Sereni~\cite{DKKS10}, who showed that any
directed graph without directed 2-edge-cut has a $\bZ_2^3 \times
\bZ_3^{9}$-\asf\ (this group has $157464$ elements).

Adding a stronger condition on the edge-connectivity allows to
prove stronger results on the size of the group $B$. It was proved by DeVos, Ne\v{s}et\v{r}il, and
Raspaud~\cite{DNR04}, that every directed 4-edge-connected graph has a  $\bZ_2^2 \times
\bZ_3^{4}$-\asf, that every directed 5-edge-connected graph has a  $\bZ_3^{5}$-\asf, and that every directed 6-edge-connected graph has a  $\bZ_2 \times
\bZ_3^{2}$-\asf.

In~\cite{Jae79}, Jaeger conjectured the
following weaker version of Tutte's 3-flow conjecture: \emph{there is a constant $k$ such that every
$k$-edge-connected graph has a $\bZ_3$-\nzf}. This conjecture was
recently solved by Thomassen~\cite{Tho12}, who proved that every 8-edge-connected
graph has a $\bZ_3$-\nzf, and was improved by Lov\'asz, Thomassen,
Wu and Zhang~\cite{LTWZ13}, that every
6-edge-connected graph has a $\bZ_3$-\nzf\ (this is a simple
consequence of Theorem~\ref{thm:boundary}). 

The natural antisymmetric variant of Jaeger's weak 3-flow
conjecture would be the following: \emph{there is a constant $k$ such that every directed
$k$-edge-connected graph has a $\bZ_5$-\asf.} 

Note that the size of the
group would be best possible, since in $\bZ_2$ and $\bZ_2\times \bZ_2$ every
element is its own inverse, while a $\bZ_3$-\asf\ or a $\bZ_4$-\asf\
has to assign the same value to all the arcs (and this is impossible in
the digraph on two vertices $u,v$ with exactly $k$ arcs
directed from $u$ to $v$, for any integer $k\equiv 1\pmod {12}$).

Our final result is the following.

\begin{thm}\label{thm:anti}
For any $k\ge 2$, every directed $\lceil \tfrac{6k}{k-1} \rceil$-edge-connected graph has a $\bZ_{2k+1}$-\asf.
\end{thm}

\begin{proof}
Let $k\ge 2$, and let $\vec{G}$ be a directed $\lceil \tfrac{6k}{k-1} \rceil$-edge-connected
graph. Let $\vec{H}$ be the directed graph obtained from $\vec{G}$ by
replacing every arc $\vec{e}$ by $k-1$ arcs with the same tail and head
as $\vec{e}$, and let $H$ be the non-oriented graph underlying
$\vec{H}$. Let
$\beta(v)=d_{\vec{G}}^-(v)-d_{\vec{G}}^+(v) $ for every $v$.
Since $\vec{G}$ is $\lceil \tfrac{6k}{k-1} \rceil$-edge-connected, $H$ is
$6k$-edge-connected and by Theorem~\ref{thm:thomexp1}, $\vec{H}$ has a $\bZ_{2k+1}$-flow $f$ with boundary
$\beta$ with flow values in the set
  $\{0, 1\} \pmod{2k+1}$. For any arc $\vec{e}$ of $\vec{G}$, let
  $g(\vec{e})$ be the sum of the values of the flow $f$ on the $t$
  arcs corresponding to $\vec{e}$ in $\vec{H}$. Then $g$ is a $\bZ_{2k+1}$-flow with boundary
$\beta$ in $\vec{G}$, with flow values in the set
$\{0,1,\ldots,k-1\}\pmod{2k+1}$. Now, set $g'(\vec{e})=g(\vec{e})+1$ for
every arc $\vec{e}$. Hence every $\vec{e}$ is
assigned a value in $\{1,\ldots,k\}\pmod{2k+1}$, and  $\partial g'(v)\equiv\partial g(v)+d_{\vec{G}}^+(v)-d_{\vec{G}}^-(v) \equiv
\beta'(v)+d_{\vec{G}}^+(v)-d_{\vec{G}}^-(v) \equiv 0 \pmod{2k+1}$ for every
$v$. Thus $g'$ is a $\bZ_{2k+1}$-flow of $\vec{G}$ with
flow values in the set $\{1,\ldots,k\}\pmod{2k+1}$, and thus a
$\bZ_{2k+1}$-\asf\ in $\vec{G}$, as desired. This concludes the proof
of Theorem~\ref{thm:anti}.
\end{proof}

As a corollary, we directly obtain:

\begin{cor}\label{cor:anti}
\mbox{}
\begin{enumerate}[(i)]
\item Every directed 7-edge-connected graph has a $\bZ_{15}$-\asf.
\item Every directed 8-edge-connected graph has a $\bZ_{9}$-\asf. 
\item Every directed 9-edge-connected graph has a $\bZ_{7}$-\asf.
\item Every directed 12-edge-connected graph has a $\bZ_{5}$-\asf.
\end{enumerate}
\end{cor}

By duality, using the results of
Ne\v{s}et\v{r}il and Raspaud~\cite{NR99}, Corollary~\ref{cor:anti}
(which, again, can be seen as an antisymmetric analogue of the
statement of Jaeger's conjecture) directly implies that every orientation of a
planar graph of girth (length of a shortest cycle) at least 12 has a homomorphism to an oriented
graph on at most 5 vertices. This was proved by Borodin, Ivanova and
Kostochka in 2007~\cite{BIK07}, and it is not known whether the same holds for planar graphs
of girth at least 11. On the other hand, it was proved by Ne\v{s}et\v{r}il, Raspaud and
Sopena~\cite{NRS97} that there are orientations of some planar graphs of
girth at least 7 that have no homomorphism to an oriented graph of at
most 5 vertices. By duality again, this implies that there are directed
7-edge-connected graphs with no $\bZ_{5}$-\asf. We conjecture the following:

\begin{conj}
Every directed 8-edge-connected graph has a $\bZ_{5}$-\asf.
\end{conj}

It was conjectured by Lai~\cite{Lai07} that for every $k\ge 1$, every
$(4k+1)$-edge-connected graph $G$ has a $\beta$-orientation for every
$\bZ_{2k+1}$-boundary $\beta$ of $G$. If true, this conjecture would
directly imply (using the same proof as that of
Theorem~\ref{thm:anti}) that for any $k\ge 2$, every directed $\lceil
\tfrac{4k+1}{k-1} \rceil$-edge-connected graph has a
$\bZ_{2k+1}$-\asf. In particular, this would show that directed
5-edge-connected graph have a $\bZ_{13}$-\asf, directed
6-edge-connected graph have a $\bZ_{9}$-\asf, directed
7-edge-connected graph have a $\bZ_{7}$-\asf, and directed
9-edge-connected graph have a $\bZ_{5}$-\asf. The bound on directed
5-edge-connected graph would also directly imply, using the proof of
the main result of~\cite{DKKS10}, that directed graphs with no
directed 2-edge-cut have a $\bZ_2^2\times \bZ_3^4\times
\bZ_{13}$-\asf.

\section{Proof of Theorem~\ref{thm:matrix}}\label{sec:matrix}

We first recall the
following (weak form of a) classical result by Mader (see \cite{Diestel}, Theorem 1.4.3):

\begin{lem}\label{lem:cut}
Given an integer $k\ge 1$, if $G=(V,E)$ is a graph with average degree at least $4k$, then
there is a subset $X$ of $V$ such that $|X|>1$ and $G[X]$ is $(k+1)$-edge-connected. 
\end{lem}

We will also need the following result of Thomassen~\cite{Tho14},
which is a simple consequence of Theorem~\ref{thm:boundary}.

\begin{thm}[\cite{Tho14}]\label{thm:thof}
Let $k \ge 3$ be an odd integer, $G=(V_1,V_2,E)$ be a bipartite graph, and  $f : V_1\cup V_2 \rightarrow \bZ_k$ be a
mapping satisfying $\sum_{v\in V_1} f(v)\equiv\sum_{v\in V_2} f(v) \pmod k$. If $G$ is $(3k- 3)$-edge-connected, then $G$ has
a spanning subgraph $H$ such that for any $v\in V$, $d_H(v) \equiv f(v) \pmod k$.
\end{thm}

Let $G$ be a graph, and let $X$ and $Y$ be two disjoint subsets of
vertices of $G$. The set of edges of $G$ with one endpoint in $X$
and the other in $Y$ is denoted by $E(X,Y)$.

We are now ready to prove Theorem~\ref{thm:matrix}.

\medskip

\noindent \emph{Proof of Theorem~\ref{thm:matrix}.}
We proceed by induction on $n$.
For $n=1$, this is a direct consequence of Theorem~\ref{thm:CD}, so suppose that
$n\ge2$.  Each basis $B_s$ can be considered as an $n\times n$ matrix
where each column is a vector with support of size at most 2. Let $\cB=\biguplus_{i=1}^{t}B_i$. 

For $1\le i \le n$, a vector is called an \emph{$i$-vector} if its support is the
singleton $\{i\}$ (in other words, the $i$-th entry is non-zero and
all the other entries are zero). Suppose that for some $1\le i\le n$, $\cB$ contains at least $p-1$ $i$-vectors. Let $\cC$ be the set of $i$-vectors of $\cB$. Clearly, each basis contains at most one $i$-vector. For every $B_s$, let $B_s'$ be the matrix obtained from $B_s$ by removing its $i$-vector (if any) and the $i^{th}$ row. Clearly $B_s'$ is or contains a basis of $\bZ_p^{n-1}$. 
By induction hypothesis, $\biguplus_{s=1}^{t}B_s'$ forms an additive
basis of $ \bZ_p^{n-1}$. In other words, for any vector
$\beta=(\beta_1,...,\beta_i,...,\beta_n)\in \bZ_p^{n}$, there is a
subset $Y_1$ of $\cB\setminus \cC$ which sums to
$(\beta_1,...,\hat{\beta}_i,..,\beta_n)$ for some $\hat{\beta}_i$. 
Since $|\cC|\ge p-1$, it follows from Theorem~\ref{thm:CD} that there is a subset $Y_2$ of $\cC$ which sums to $(0,...,\beta_i-\hat{\beta}_i,..,0)$. Hence $Y_1\cup Y_2$ sums to $\beta$. 

Thus we can suppose that there are at most $p-2$ $i$-vectors for every
$i$. Then there are at least $8 \ell (3p-4)n$ vectors with a support
of size 2 in $\cB$. Since there are at most $\ell$ distinct
shadows of size 2 in $\cB$, there are at least  $8(3p-4)n$ vectors with the same
(unordered) shadow of size 2, say $\{a_1,a_2\}$ (recall that shadows
are multisets, so $a_1$ and $a_2$ might coincide).

Let $G$ be the graph (recall that graphs in this paper are allowed to
have multiple edges) with vertex set $V=\{v_1,...,v_n\}$ and
edge set $E$, where edges $v_iv_j$ are in
one-to-one correspondence with vectors of $\cB$ with support $\{i,j\}$ and shadow $\{a_1,a_2\}$.
Then $G$ contains at least $8(3p-4)n$ edges.

We now consider a random partition of $V$ into 2 sets $V_1,V_2$
(by assigning each vertex of $V$ uniformly at random to one of the
sets $V_k$, $k=1,2$). Let $e=v_iv_j$ be some edge of $G$. Recall that
$e$ corresponds to some vector with only two non-zero entries, say
without loss of generality $a_1$ at $i^{th}$ index and $a_2$ at
$j^{th}$ index. The probability that $v_i$ is assigned to $V_1$ and
$v_j$ is assigned to $V_2$ is at least $\tfrac14$. As a consequence, there is a partition of $V$ into 2
sets $V_1,V_2$ and a subset $E'\subseteq E(V_1,V_2)$ of at least $
8(3p-4)n/4=2(3p-4)n$
edges such that
 for every $e\in E'$,  
the vector of $\mathcal{B}$ corresponding with $e$ has entry $a_1$ (resp. $a_2$) at the index associated to the endpoint of $e$ in $V_1$ (resp. $V_2$).

Since the graph $G'=(V,E')$ has average degree at
least $4(3p-4)$, it follows from
Lemma~\ref{lem:cut} that there is a set $X\subseteq V$ of at least 2
vertices, such that $G'[X]$ is
$(3p-3)$-edge-connected. 
Set $H=G'[X]$ and $F$ the edge set of $H$.
%In the remainder we write $H$ instead of
%$G'[X]$ and denote the edge set of $H$ by $F$. 
Note that $H$ is bipartite with bipartition
 $X_1=X\cap V_1$ and $X_2=X\cap V_2$.

For each integer $1\le s \le t$, let $B_s^*$ be the matrix obtained from $B_s$ by doing
the following: for each vertex $v_i$ in $X_1$ (resp. $X_2$),
we multiply all the elements of the $i^{th}$ row of $B_s$ by
$a_1^{-1}$ (resp. $-a_2^{-1}$), noting that all the operations are performed in $\bZ_p$. Let $\cB^*=\biguplus_{s=1}^{t}B_s^*$. Note that each vector of
$\cB^*$ corresponding to some edge
$e\in F$ has shadow $\{1,-1\}$ ($1$ is the entry
indexed by the endpoint of $e$ in $X_1$ and $-1$ is the entry
indexed by the endpoint of $e$ in $X_2$). It is easy to verify the following.

\begin{itemize}
\item Each $B_s^*$ is a linear basis of $\bZ_p^n$.
\item $\cB$ is an additive basis if and only if $\cB^*$ is an additive basis.
\end{itemize}
Hence it suffices to prove that  $\cB^*$ is an additive basis.

Without loss of generality, suppose that $X=\{v_m,...,v_n\}$ for some $m\le n-1$. 
By \emph{contracting} $k$ rows of a matrix, we mean deleting these $k$
rows and adding a new row consisting of the sum of the $k$ rows. For each
$1\le s \le t$, let $B_s'$ be the matrix of $m$ rows obtained from $B_s^*$ by contracting all $m^{th},(m+1)^{th},...,n^{th}$ rows. Note that
the operation of contracting $k$ rows decreases the rank of the matrix
by at most $k-1$ (since it is the same as replacing one of the rows by the
sum of the $k$ rows, which preserves the rank, and then deleting the
$k-1$ other rows). Let $\cB'=\biguplus_{s=1}^{t}B_s'$ . Since each $B_s^*$ is a linear basis of $\bZ_p^n$,
each $B_s'$ has rank at least $m$ and therefore contains a basis of $\bZ_p^m$.
Hence, by induction hypothesis,
$\cB'\setminus \cB_0'$ is an additive
basis of $\bZ_p^m$, where $\cB_0'$ is the set of all columns with
empty support in $\cB'$. For every
$\beta=(\beta_1,...,\beta_n)\in \bZ_p^n$, let
$\beta'=(\beta_1,...,\beta_{m-1}, \sum_{i=m}^{n}\beta_i)\in
\bZ_p^m$. Then there is a subset $Y'$ of
$\cB'\setminus \cB_0'$ which sums to
$\beta'$. Let $Y^*$ and $\cB_0^*$ be the subsets of $\cB^*$ corresponding to
$Y'$ and $\cB_0'$, respectively. Then $Y^*$ sums to some
$\hat{\beta}=(\beta_1,...,\beta_{m-1},\hat{\beta}_m,...,\hat{\beta}_n
)$, where $\sum_{i=m}^{n}\hat{\beta}_i\equiv
\sum_{i=m}^{n}\beta_i\pmod p$.

Recall that for each edge $e\in F$, the corresponding vector in
$\cB^*$ has precisely two non-zero entries, $(1,-1)$, each with
index in $X$. Hence the vector corresponding to each $e\in F$ in
$\cB'$ has empty support. Thus the set of vectors in
$\cB^*$ corresponding to the edge set $F$ is a subset of $\cB_0^*$, which
is disjoint from $Y$.

For each $v_i\in X_1$, let $\beta_X(v_i)=\beta_i-\hat{\beta}_i$,
and for each $v_i\in X_2$, let $\beta_X(v_i)=\hat{\beta}_i
-\beta_i$. Since
$\sum_{i=m}^{n}\hat{\beta}_i\equiv\sum_{i=m}^{n}\beta_i\pmod p$,
we have $\sum_{v_i\in X\cap V_1} \beta_X(v_i)=\sum_{v_i\in X\cap V_2}
\beta_X(v_i)$. Since $H$ is $(3p-3)$-edge-connected, it follows from
Theorem~\ref{thm:thof} that there is a subset $F'\subseteq F$ 
such that, in the graph $(X,F')$, each vertex $v_i \in X_1$ has
degree  $\beta_i-\hat{\beta}_i\pmod p$ and each vertex $v_i \in X_2$ has
degree  $\hat{\beta}_i-\beta_i\pmod p$. Therefore, $F'$
corresponds to a subset $Z^*$ of vectors of $\cB_0^*$, summing to
$(0,\ldots,0,
\beta_m-\hat{\beta}_m,\ldots,\beta_n-\hat{\beta}_n)$. Then $Y^*\cup
Z^*$ sums to $\beta$. It follows that $\cB^*$ is an additive basis of
$\bZ_p^n$, and so is $\cB$. This completes the proof.\hfill $\Box$

\section{Two proofs of (versions of) Theorem~\ref{thm:lyon}}\label{sec:flow}

We now give two proofs of (versions of) Theorem~\ref{thm:lyon}. The first one is a
direct application of Corollary~\ref{cor:zerosum}, but requires a stronger
assumption on the edge-connectivity of $G$ ($24p^2-54p+28$
instead of $6p^2-14p+8$ for the second proof).

\medskip

\noindent \emph{First proof of Theorem~\ref{thm:lyon}.}
We fix some arbitrary orientation
$\vec{G}=(V,\vec{E})$ of $G$ and denote the vertices
of $G$ by $v_1,\ldots,v_n$. The number of edges of $G$ is denoted by
$m$. For each arc $\vec{e}=(v_i,v_j)$ of
$\vec{G}$, we associate $\vec{e}$ to a vector $x_e \in (\bZ_p^n)_0$ in which the
$i^{th}$-entry is equal to $f(e)\pmod p$, the $j^{th}$-entry is equal to
$-f(e)\pmod p$ and all the
remaining entries are equal to $0\pmod p$.

Let us consider the following statements.
\begin{enumerate}[(a)]
\item \label{p1} For each
$\bZ_p$-boundary $\beta$, there is an $f$-weighted $\beta$-orientation
of $G$.
\item \label{p2}
For each
$\bZ_p$-boundary $\beta$ there is a vector $(a_e)_{e\in E} \in
\{-1,1\}^m$, such that $\sum_{e\in E} a_e x_e\equiv \beta \pmod
p$. 
\item  \label{p3} For each
$\bZ_p$-boundary $\beta$ there is a vector $(a_e)_{e\in E} \in
\{0,1\}^m$ such that $\sum_{e\in E} 2 a_e
x_e\equiv \beta \pmod
p$.
\end{enumerate}
Clearly, \ref{p1} is equivalent to \ref{p2}. We now claim that \ref{p2} is equivalent to 
\ref{p3}.
To see this, simply do the following for each arc $\vec{e}=(v_i,v_j)$ of
$\vec{G}$: add $f(e)$ to the $j^{th}$-entry of $x_e$ and to
$\beta(v_j)$, and subtract $f(e)$ from the $i^{th}$-entry of $x_e$ and from
$\beta(v_i)$. 
To deduce \ref{p3} from Corollary~\ref{cor:zerosum}, what is left is to show that $\{ a_e:e\in E\}$ can be decomposed into
sufficiently many linear bases of $(\bZ_p^n)_0$. This follows from the
fact that $G$ is $(8(p-1)(3p-4)+2p-4)$-edge-connected (and therefore contains $4(p-1)(3p-4)+p-2$
edge-disjoint spanning trees) and that the set of vectors $a_e$
corresponding to the edges of a spanning tree of $G$
forms a linear basis of $(\bZ_p^n)_0$ (see~\cite{Jae92}).\hfill $\Box$

\medskip

A second proof consists in
mimicking the proof of Theorem~\ref{thm:matrix} (it turns out to give a
better bound for the edge-connectivity of $G$). 

\medskip

\noindent \emph{Second proof of Theorem~\ref{thm:lyon}.}
As before, all values and operations are considered modulo $p$.
We can assume without loss of generality that $f(E)\in
\{1,2,\ldots,\tfrac{p-1}2\}$, since otherwise we can replace the value $f(e)$ of
an edge $e$ by $-f(e)$, without changing the problem. 

We prove the result by induction on the number of vertices of $G$.
The result is trivial if $G$ contains only one vertex, so assume that
$G$ has at least two vertices.

For any
$1\le i \le k$, let $E_i$ be the set of edges $e\in E$ with $f(e)=i$,
and let $G_i=(V,E_i)$. Since $G$ is $(6p-8)(p-1)$-edge-connected, $G$ has
minimum degree at least $(6p-8)(p-1)$ and then average degree at least
$(6p-8)(p-1)$. As a consequence, there exists $i$ such that $G_i$ has average
degree at least $12p-16$. By Lemma~\ref{lem:cut}, since
$\tfrac{12p-16}{4}+1=3p-3 $, $G_i$ has
an induced subgraph $H=(X,F)$ with at least two vertices such that $H$ is
$(3p-3)$-edge-connected. Let $G/X$ be the graph obtained
from $G$ by contracting $X$ into a single vertex $x$ (and removing
possible loops). Since $H$ contains more than one vertex,
$G/X$ has less vertices than $G$ (note that possibly,
$X=V$ and in this case $G/X$ consists of the single vertex $x$). Since
$G$ is $(6p-8)(p-1)$-edge-connected, $G/X$ is also $(6p-8)(p-1)$-edge-connected. Hence by the induction hypothesis it has an $f$-weighted
$\beta$-orientation, where we consider the restriction of $f$ to the
edge-set of $G/X$, and we define $\beta(x)=\beta(X)$. Note that this
orientation corresponds to an orientation of all the edges of $G$ with at
most one endpoint in $X$.

We now orient arbitrarily the edges of $G[X]$ not in $F$ (the edge-set
of $H$), and update the values of the $\bZ_{p}$-boundary
$\beta$ accordingly (i.e. for each $v \in X$, we subtract from $\beta(v)$
the contribution of the arcs that were already oriented). It is easy to see that as the
original $\beta$ was a boundary, the new $\beta$ is indeed a boundary. Finally,
since all the edges of $H$ have the same weight, and since $H$ is $(3p-3)$-edge-connected, it
follows from Corollary~\ref{cor:thomexp2} that $H$ has an $f$-weighted
$\beta$-orientation (with respect to the updated boundary $\beta$). The
orientations combine into an $f$-weighted $\beta$-orientation of $G$, as desired.
\hfill $\Box$

\section*{Acknowledgments}

We would like to thank the referees for their suggestions and for
mentioning the existence of the unpublished manuscript~\cite{LaiLi}.


\begin{thebibliography}{99}

\bibitem{ALM91} N. Alon, N. Linial and R. Meshulam
\emph{Additive bases of vector spaces over prime fields},  J. Combin. Theory Ser. A {\bf 57} (1991), 203--210.

\bibitem{ANR96} N. Alon, M. Nathanson, and I. Ruzsa, \emph{The
    polynomial method and restricted sums of congruence classes},
  J. Number Theory {\bf 56(2)} (1996), 404--417.

\bibitem{BM76} J.A.~Bondy and U.S.R.~Murty, \emph{Graph Theory
    with Applications}, Macmillan, London and Elsevier, New York,
    1976.

\bibitem{BIK07} O.V. Borodin, A.O. Ivanova and A.V. Kostochka,
\emph{Oriented 5-coloring of sparse plane graphs},  J. Applied and
Industrial Math. {\bf 1(1)} (2007), 9--17.

\bibitem{Dav35} H. Davenport,  \emph{On the addition of residue
    classes}, J. London Math. Soc. {\bf 10} (1935), 30--32.

\bibitem{DeV00} M. DeVos, \emph{Matrix choosability}, J. Combin. Theory Ser. A {\bf 90} (2000),
197--209.

\bibitem{DJS} M. DeVos, T. Johnson, and P. Seymour, \emph{Cut coloring and
    circuit covering}, Manuscript.

\bibitem{DNR04} M. DeVos, J. Ne\v{s}et\v{r}il, and A. Raspaud,
  \emph{Antisymmetric flows and edge-connectivity}, Discrete
  Math. {\bf 276(1–3)} (2004), 161--167.
  
  \bibitem{Diestel} R. Diestel, Graph Theory, \emph{Graduate Texts in Mathematics, Springer} (2005).

\bibitem{DKKS10} Z. Dvo\v{r}\'ak, T. Kaiser, D. Kr\'al', and
  J.-S. Sereni, \emph{A note on antisymmetric flows in graphs},
  European J. Combin. {\bf 31} (2010), 320--324.

\bibitem{Jae79} F. Jaeger, \emph{Flows and generalized coloring
    theorems in graphs}, J. Combin. Theory Ser. B {\bf 26} (1979), 205--216.

%\bibitem{Jae88} F.~Jaeger, \emph{Nowhere-zero flow problems}, in:
% L. Beineke, et al. (Eds.), Selected Topics in Graph Theory, vol. 3,
% Academic Press, London, New York, 1988, pp. 91--95.
 
\bibitem{Jae92} F. Jaeger, N. Linial, C. Payan, and M. Tarsi, \emph{Group connectivity of graphs -- A non homomogeneous analogue of nowhere-zero flow properties}, J. Combin. Theory Ser. B {\bf 56}
  (1992), 165--182.

\bibitem{Lai07} H.-J. Lai, \emph{Mod $(2p + 1)$-orientations and
    $K_{1,2p+1}$-decompositions}, SIAM J. Discrete Math. {\bf 21}
  (2007), 844--850.

\bibitem{LaiLi} H.-J. Lai and P. Li, \emph{Additive bases and strongly
    $\bZ_{2s+1}$-connectedness}, Manuscript.

\bibitem{LTWZ13} L.M. Lov\'asz, C. Thomassen, Y. Wu, and C.-Q. Zhang,
  \emph{Nowhere-zero 3-flows and modulo $k$-orientations},
  J. Combin. Theory Ser.B {\bf 103} (2013), 587--598.

\bibitem{NR99} J. Ne\v{s}et\v{r}il and A. Raspaud,
  \emph{Antisymmetric flows and strong colourings of oriented graphs},
  Ann. Inst. Fourier {\bf 49(3)} (1999), 1037--1056.

\bibitem{NRS97} J. Ne\v{s}et\v{r}il, A. Raspaud and E. Sopena,
  \emph{Colorings and girth of oriented planar graphs}, Discrete
 Math. {\bf 165--166} (1997), 519--530.

\bibitem{Ste76} R. Steinberg, \emph{Gr\"otzsch's Theorem dualized},
  M. Math Thesis, University of Waterloo, 1976. 

\bibitem{Tho12} C. Thomassen, \emph{The weak 3-flow conjecture and the
    weak circular flow conjecture}, J. Combin. Theory Ser. B {\bf 102} (2012),
521--529.

\bibitem{Tho14} C. Thomassen, \emph{Graph factors modulo $k$},
  J. Combin. Theory Ser. B {\bf 106} (2014), 174--177.

\bibitem{Tut54} W.T.~Tutte, \emph{A Contribution on the Theory of
Chromatic Polynomial}, Canad. J. Math. {\bf 6} (1954), 80--91.
%
\bibitem{Tut66} W.T.~Tutte, \emph{On the algebraic theory of graph
 colorings}, J. Combin. Theory {\bf 1} (1966), 15--50.

%% \bibitem{zhang} C.Q. Zhang, \emph{Integer Flows and Cycle Covers of Graphs}, New
%% York, Marcel Dekker, 1997.

\end{thebibliography}
\end{document}